\documentclass[11pt,openbib]{amsart}
\usepackage{amssymb,amsmath,amsthm}
\usepackage{fancyhdr}

\newcommand{\nn}{\mathbb{N}}
\newcommand{\zz}{\mathbb{Z}}
\newcommand{\qq}{\mathbb{Q}}

\newcommand{\pp}{\mathbb{P}}
\newcommand{\ed}{{\rm ed }}
\newcommand{\cd}{{\rm cd }}
\newcommand{\rank}{{\rm rank}\;}
\newcommand{\trdeg}{{\rm trdeg}_k }
\theoremstyle{plain}
\newtheorem{prop}{Proposition}[section]
\newtheorem{thm}{Theorem}[section]
\newtheorem{lem}{Lemma}[section]
\newtheorem{cor}{Corollary}[section]

\theoremstyle{definition}
\newtheorem{defn}{Definition}[section]

\newtheorem*{exs}{Examples}
\theoremstyle{remark}
\newtheorem*{rem}{Remark}

\begin{document}
\title{A Valuation Theoretic Approach to Essential Dimension}
\author{Aurel Meyer\\	
	Department of Mathematics \\
	University of British Columbia\\
	Vancouver, BC  \\Canada V6T1Z2\\}

\begin{abstract}
In this essay we explore the notion of essential dimension using the theory of valuations of fields. Given a field extension $K/k$ and a valuation on $K$ that is trivial on $k$, we prove that the rank of the valuation cannot exceed the transcendence degree $\trdeg K$. We use this inequality to prove lower bounds on the essential dimension in some interesting situations. We study orbits of a torus action and find a formula for the essential dimension of the functor of these orbits.


\end{abstract}
\maketitle

\section{Introduction}
Throughout this paper, the ground field $k$ is assumed to be an algebraically closed field of characteristic $0$. 
Let $\mathcal{K}$ be the category of field extensions $K$ of $k$, $\mathcal{C}$ the category of sets and 
\[
\mathcal{F}:\mathcal{K}\rightarrow\mathcal{C}
\]
a covariant functor.
\begin{exs}
1)\quad The functor of elliptic curves. We denote it by $\mathcal{F}_{El}$. It assigns to each field $K\in\mathcal{K}$ the set of all elliptic curves defined over $K$.\\
2)\quad The functor $\mathcal{F}_{m,d}$ of homogeneous forms in a rigid system. Let $f$ be a homogeneous $d-$form in $K^m$ and $L_1,...,L_m$ lines in general position. Define $\mathcal{F}_{m,d}(K)$ to be the set of equivalence classes $[f,L_1,...,L_m]$ where equivalence is given by an isomorphism on $K^m$ that preserves the lines and the form. \\
3)\quad The functor of orbits, $\mathcal{F}_{Orb}$. Let $X$ be a variety defined over $k$ and $G$ be an algebraic group acting on $X$. Then $\mathcal{F}_{Orb}(K)$ is set of all $G(K)$-orbits in $X(K)$.\\
4)\quad  Let $G$ be an algebraic group. Define the $G$-functor $\mathcal{F}_G(K)=H^1(K,G)$ as the first Galois cohomology set.
\end{exs}
One might be interested in the minimal number of independent parameters needed to define an object in $\mathcal{F}(K)$ or the minimal number of independent parameters needed to define the structure given by the functor $\mathcal{F}$. Essential dimension is a measure for these numbers. As usual, essential dimension will be denoted by $\ed()$. The following definition is due to A. Merkurjev (unpublished) and can be found in \cite{BF}:

For an object $\alpha\in\mathcal{F}(K)$ define
\begin{defn}\label{ED-O}
$\ed(\alpha)=\min\{\trdeg(K_0)\mid \alpha \in \mbox{im}(\mathcal{F}(K_0)\rightarrow\mathcal{F}(K))\}$
where $K_0$ ranges over field extension $k\subset K_0 \subset K$ and $\mathcal{F}(K_0)\rightarrow\mathcal{F}(K)$ is a morphism.
\end{defn}\label{ED-F}
The essential dimension of the functor $\mathcal{F}$ is defined to be
\begin{defn}
$\ed(\mathcal{F})=\max\{\ed(\alpha)\mid K \in \mathcal{K}, \alpha \in \mathcal{F}(K)\}$.
\end{defn}

Essential dimension was first introduced by J. Buhler and Z. Reichstein in \cite{BR1} for finite groups and by Z. Reichstein for algebraic groups in \cite{Re1}. 
Since then, many mathematicians have studied essential dimension in different contexts.
In this essay we will consider all of the functors from the examples above and compute the essential dimension in particular cases.

In the first sections our focus will be on valuations of fields. Let $K\in \mathcal{K}$ be a field extension of $k$. A valuation
\[
\upsilon: K^*\rightarrow G
\] 
is a group homomorphism onto the finitely generated, abelian and ordered group $G$ (called the valuation group). We assume that $\upsilon$ is trivial on $k$.
In Theorem \ref{mthm} we will prove, that for any valuation the inequality 
\[
\trdeg K\ge \rank G
\]
holds. This in turn can be used to prove lower bounds on the essential dimension of objects in $\mathcal{F}(K)$. Indeed, given an object $\alpha \in \mathcal{F}(K)$, one can try to construct a valuation $\upsilon: K^*\rightarrow G$ such that for any field $K_0\subset K$ with $\alpha\in\mbox{im}(\mathcal{F}(K_0)\rightarrow\mathcal{F}(K))$ the induced valuation $\upsilon|_{K_0}: K_0^*\rightarrow \upsilon(K_0^*)$ is of a certain, preferably high, rank. The choice of the valuation is of course influenced by the specific structure of the functor $\mathcal{F}$.
 
In section \ref{val}, assuming that $K$ is the function field of a variety $X$, we show a method of constructing valuations 
\[
\upsilon:K^*=k(X)^*\rightarrow\zz^r
\]
through chains of divisors (hypersurfaces) 
\[
X=H_0\supset H_1\supset...\supset H_r.
\]
The main example is rational functions in $n$ variables. Let $X=\mathbb{A}^n$ so that $K=k(X)=k(x_1,...,x_n)$ and the divisors be intersections of the coordinate hyperplanes: 
\[
H_1=\{x_1=0\},\;H_2=\{x_1=x_2=0\},...,H_n=\{x_1=...=x_n=0\}.
\]
The valuation one gets this way, assigns to every function its order of vanishing at $x_1=0,x_2=0$, etc. i.e. is defined by $\upsilon(x_i)=e_i\in\zz^n$.
This simple valuation often suffices our purpose.

In section \ref{ota} we study the orbits of a torus action on a variety. 
For a field $K/k$, the torus $T=(K^*)^m$ acts on $K^n$ via characters 
\[
(t\cdot a)_i=\chi_i(t)a_i\quad i=1,...,n\quad\mbox{where}
\]
\[ 
\chi_i(t)=\chi_i(t_1,...,t_m)=t_1^{e_{i1}}\cdots t_m^{e_{im}}\in K^*.
\] 
All information about this action is stored in the exponent matrix 
\[
E=\left(\begin{array}{ccc}e_{11}&\cdots&e_{1m}\\
\\\vdots&&\vdots\\\\
e_{n1}&\cdots&e_{nm}\end{array}\right)\in M_{n,m}(\zz)
\]
We write $\mathcal{F}_E$ for the functor of orbits of this action. We will determine the essential dimension of $\mathcal{F}_E$ with the valuation theoretic approach and write it in terms of the elementary divisors $d_1,...,d_r$ of the matrix $E$:
\[
\ed(\mathcal{F}_E)=n-\mbox{\# ones among }\{d_1,...,d_r\}
\]
(Theorem \ref{edgen} and Theorem \ref{edF}).
As important examples we calculate the essential dimension of elliptic curves,
\[
\ed(\mathcal{F}_{El})=2 \quad\mbox{(Corollary \ref{elc})}
\]
and of rigid homogeneous $d-$forms in $m$ variables (section \ref{hf}): 
\[
\ed(\mathcal{F}_{m,d})=
\left\{\begin{array}{ll}{{m+d-1}\choose d}-m+1&d>1\\0&d=1\end{array}\right..
\] 

In section \ref{ps} we consider the torus action in projective space. The functor of orbits there we denote by $\mathcal{PF}_E$ and show that 
\[
\ed(\mathcal{F}_E)\ge\ed(\mathcal{PF}_E)\ge\ed(\mathcal{F}_E)-1.\quad\mbox{(Proposition \ref{F-PF})}
\]
An example here are hypersurfaces in projective space.

In section \ref{FAG} we shift our attention to finite abelian groups. The essential dimension of an algebraic group $G$ is defined to be the essential dimension of the functor $\mathcal{F}_G$ (example 4). We will also recall the original definition of essential dimension of finite groups as in \cite{BR1}. With the valuation theoretic approach we prove that the essential dimension of finite abelian groups is its rank, a result that was proved with different methods in \cite{BR1}, \cite[Example 7.4]{RY} or \cite[Prop. 3.7]{BF} .

In the last section we will consider another numerical invariant, the canonical dimension of an algebraic group (after G. Berhuy and Z. Reichstein, \cite{BR3}) and relate some earlier results to the notion of canonical dimension.

\subsection*{Acknowledgments}
I thank my supervisor Zinovy Reichstein for providing me with ideas and for helpful comments and discussions.

\section{Preliminaries}
In this section we recall the necessary definitions and results from valuation theory and we will show how to construct a valuation of the function field of an algebraic variety through a chain of divisors. For more details on valuations we refer to \cite{HP}, \cite{AM} and \cite{La}.
\subsection{Valuations}
Let $R$ be an integral domain and $K$ its field of fractions. $R$ is called a \textit{valuation ring} of $K$ if for each $x\in K^*, x$ or $x^{-1}$ is in $R$ (or both).\\
Let $\mathfrak{m}$ be the set of non-units of $R$ and $U$ be the units.
\begin{prop}
$R$ is an integrally closed local ring with unique maximal ideal $\mathfrak{m}$.
\end{prop}
\begin{proof}
\cite[Prop 5.18]{AM}.
\end{proof}

Let $G$ be a finitely generated Abelian group with an ordering "$>$", i.e. \\
if $g,h\in G, g,h>0\Rightarrow 0>-g$ and $g+h>0$.\\
Note that $G$ has no element of finite order hence $G\cong \zz^n$ for some $n$.
\begin{defn}
For a field $K$, a \textit{valuation} on $K$ is a surjective map 
$$\upsilon:K^*\rightarrow G\quad \mbox{such that}$$
i)\quad$\upsilon(xy)=\upsilon(x)+\upsilon(y)$\\
ii)\quad$\upsilon(x+y)\ge\min\{\upsilon(x),\upsilon(y)\}$\quad if $x+y\ne 0$.
\end{defn}
\noindent In ii) equality holds if $\upsilon(x)\ne\upsilon(y)$. $\upsilon(0)$ is sometimes set to infinity.
The set 
\[
R=\{x\in K^*\mid\upsilon(x)\ge0\}\cup\{0\}
\]
is a valuation ring with maximal ideal 
\[
\mathfrak{m}=\{x\mid\upsilon(x)>0\}.
\]
In what follows, we assume that all fields are extensions $K/k\in\mathcal{K}$ and the valuations are trivial on $k$, $\upsilon(k^*)=0$.
If we have a field extension $K\supset F\supset k$, then a valuation $\upsilon$ on $K$ induces a valuation 
\[
\upsilon|_F:F^*\rightarrow\upsilon(F^*)
\]
and 
$\upsilon(F^*)$ is an ordered subgroup of $G$. 
\begin{defn}
The \textit{rank} of the valuation $\upsilon$ is defined to be the rank of the valuation group $G$.
\end{defn}
\noindent The rank of a valuation will play an important role in what follows. \\
For any valuation $\upsilon$ with valuation ring $R$ and maximal ideal $\mathfrak{m}$ one has a field $R/\mathfrak{m}$ which is called the \textit{residue field} of the valuation.\\
A valuation $\upsilon$ is called \textit{discrete} if the valuation group $G$ is isomorphic to $\zz$. Its associated valuation ring $R$ is called a \textit{discrete valuation ring}.\\
Now let $\upsilon: K^*\rightarrow\zz$ be a discrete valuation.
Since $\upsilon$ is surjective, there is a $\pi\in\mathfrak{m}^*$ with $\upsilon(\pi)=1$. 
\begin{defn}An element $\pi\in K^*$ with $\upsilon(\pi)=1$ is called a
\textit{local} or \textit{uniformizing parameter} of the valuation.
\end{defn}
\noindent Every $f\in \mathfrak{m}^*$ can be written as 
$$f=\pi^ku$$
with unique $k\in\zz$ and $u\in U$ a unit in $R$ (\cite[12.6]{La}). One could say $f$ has a pole ($k<0$) or a singularity ($k>0$) of order $k$ at $\pi=0$. Clearly $\mathfrak{m}=(\pi)$. It can be proved that $R$ is Noetherian and that $\mathfrak{m}$ is the only prime ideal of $R$, in particular $R$ is of dimension $1$. 
\begin{prop}\label{vr}
Let $R$ be a Noetherian local ring of dimension $1$. Then $R$ is a discrete valuation ring if and only if $R$ is integrally closed in its field of fractions.
\end{prop}
\begin{proof}
\cite[Prop 9.2]{AM}.
\end{proof}

\subsection{Valuations of Function Fields}\label{val}
Now we will construct valuations of the function field of a variety. We will see that any prime divisor of the variety gives rise to a discrete valuation and we can repeat the construction to get valuations of higher ranks.\\
Let $X$ be a normal irreducible algebraic variety. Assume first that $X$ is affine. Let
$k[X]$ be its coordinate ring.
An irreducible subvariety $H\subset X$ of codimension $1$ is called a \textit{prime divisor} (or \textit{hypersurface}). Let
\[
I(H)=\{f\in k[X]\mid f(x)=0 \;\forall\; x\in H\}\quad\mbox{its ideal,}
\]
\[
\mathcal{O}_{X,H}=k[X]_{I(H)}
\]
the local ring of $H$ of all rational functions defined at some point of $H$.\\
Note that since $H$ is irreducible, $I(H)$ is a prime ideal in $\mathcal{O}_{X,H}$ and 
$\dim\mathcal{O}_{X,H}=\dim X-\dim H=1$.
By the assumption that $X$ is normal, $\mathcal{O}_{X,H}$ is integrally closed in $k(X)$ (see
\cite[Lemma II.5.1]{Sha}).
It follows from Proposition \ref{vr} that $\mathcal{O}_{X,H}$ is a valuation ring for a valuation
\[
\upsilon_H:k(X)^*\rightarrow k(X)^*/U\cong\zz
\]
where $U$ denotes the units in $\mathcal{O}_{X,H}$. We can choose a uniformizing parameter $\pi\in I(H)^*$ such that
\[
\begin{array}{c}
\upsilon_H(\pi)=1\\
\upsilon_H(f)>0\iff f\in I(H)^*\\
\upsilon_H(f)=0\iff f\in U
\end{array}
\]
We can see that the residue field of this valuation is isomorphic to the function field of $H$:
\begin{equation}\label{iso}
\mathcal{O}_{X,H}/I(H)\cong U|_H\cup\{0\}\cong k(H)
\end{equation}
Note that $I(H)=(\pi)$ and $H$ can be replaced by $\{\pi=0\}\subset X$.\\
Now let $X$ be any (quasi-projective) variety. We cannot expect to find a unique local parameter that cuts out the divisor.
Here we choose an open normal and affine set $Y\subset X$  such that it intersects $H$. Then we get a valuation
\[
\upsilon_{H\cap Y}:k(X)^*=k(Y)^*\rightarrow\zz
\]
This construction depends on $Y$, so consider another open and affine subset $Y^\prime$ that intersects $H$. First, if $Y^\prime\subset Y$ and $\pi$ is a uniformizing parameter in $Y$, then so it is in $Y^\prime$. The valuations obtained from $Y$ and $Y^\prime$ can therefore differ by at most the choice of a different uniformizing parameter. If $Y^\prime$ is arbitrary, then, since $H$ is irreducible, $Y\cap Y^\prime\neq\emptyset$. Choose an affine open neighbourhood $Z$ of a point $x\in Y\cap Y^\prime\cap H$ and it follows that the valuations defined through $Z$, $Y$ and $Y^\prime$ all differ by at most the choice of a different uniformizing parameter. Proposition \ref{up} will give a complete answer how much that is.\\

Next we want to repeat this construction and get a valuation
\[
\upsilon:k(X)^*\rightarrow\zz^r
\]
Assume we have a chain
\[
X=H_0\supset H_1\supset...\supset H_r\quad(r\le \dim X)
\]
of irreducible subvarieties such that each $H_i$ is of codimension $1$ in $H_{i-1}$ and $H_0,...,H_{r-1}$ are normal.
Given a collection of uniformizing parameters $\pi_1,...,\pi_r$ we get valuations
\[
\upsilon_{H_i}:k(H_{i-1})^*\rightarrow\zz
\]
with valuation rings $\mathcal{O}_{H_{i-1},H_i},\,i=1,...,r$.
By \ref{iso} identify $k(H_{i-1})$ with the residue field of $\upsilon_{H_i}$.
Then a rational function $f\in k(X)^*$ can be written uniquely 
\[
f=\pi_1^{k_1}\cdots\pi_r^{k_r}u\quad k_1,...,k_r\in\zz, \upsilon_{H_r}(u)=0
\]
where the residue of $\pi_i^{k_i}\cdots\pi_r^{k_r}u$ is a rational function in $k(H_{i-1})^*$ and \\ $k_i=\upsilon_{H_i}(\pi_i^{k_i}\cdots\pi_r^{k_r}u|_{H_{i-1}})$.
The map
\[
\begin{array}{rcl}
\upsilon:k(X)^*&\rightarrow &\zz^r \\
f&\mapsto&(k_1,...,k_r)
\end{array}
\]
is a valuation of rank $\le r$.\\
Note: $\upsilon$ depends on the choice of the divisors $H_1,...,H_r$ as well as on the choice of the uniformizing parameters $\pi_1,...,\pi_r$.

\subsection{Example}\label{mex}
Let $X=\mathbb{A}^n$. We can simply take $H_1=\{x_1-c_1=0\},$\\$H_2=\{x_1-c_1=x_2-c_2=0\},...,H_n=\{x_1-c_1=...=x_n-c_n=0\}$ as our divisors, where $x_1,...,x_n$ are the coordinate functions and $c_1,...,c_n\in k$. For uniformizing parameters we choose $\pi_i=x_i-c_i$. A rational function $f$ can be written as 
\[
f=(x_1-c_1)^{k_1}\cdots(x_n-c_n)^{k_n}u
\]
where $u$ is not divisible by any of the $(x_i-c_i)$. The valuation just assigns to $f$ the order of vanishing at $x_i=c_i,\;i=1,..,n$. For a variety $Y$ whose function field is a subfield of $k(X)$ one can of course restrict this valuation to $k(Y)$. This way of constructing a valuation will be often used. \\
Note: If $c_1=...=c_n=0$ and $f$ is a Laurent polynomial $f\in k[x_1^{\pm 1},...,x_n^{\pm 1}]$, the valuation gives the lowest exponent $(e_1,...,e_n)\in\zz^n$ (with respect to the lexicographic order of $\zz^n$) occurring in $f$. The highest exponent is often denoted by \textbf{in}$(f)$ and is of interest for example studying Gr\"obner and SAGBI bases. 
We have
\[
\textbf{in}(f)=-\upsilon(f(x_1^{-1},...,x_n^{-1}))
\]

\subsection{Uniformizing Parameters}
Now what happens if we choose different uniformizing parameters, say $\pi_1^\prime,..,\pi_r^\prime$?
Let $f\in k(X)^*$ be a rational function,
\[
f=\pi_1^{k_1}\cdots\pi_r^{k_r}u=\pi_1^{\prime\,k_1^\prime}\cdots\pi_r^{\prime\, k_r^\prime}u^\prime
\]
We are interested in how the two valuations 
$\upsilon(f)=(k_1,...,k_r)$ and $\upsilon^\prime(f)=(k_1^\prime,...,k_r^\prime)$ differ.
\begin{prop}\label{up}
There is a upper triangular matrix $A\in\mbox{GL}_r(\zz)$ with ones on the diagonal, such that 
$\upsilon^\prime(f)=\upsilon(f)\cdot A$ for all $f\in k(X)^*$.
\end{prop}
\begin{proof}
Define $A=(A_{i,j})$ by 
\[
\begin{array}{ll}
A_{i,j}=0&j<i\\
A_{i,i}=1&\\
A_{i,j}=\upsilon_{H_j}\left(\pi_i\cdot\prod_{l=1}^{j-1}(\pi_l^\prime)^{-A_{i,l}}|_{H_{j-1}}\right)&j=i+1,...,r
\end{array}
\]
To see that this is well defined we need to show that $\pi_i\cdot\prod_{l=1}^{j-1}(\pi_l^\prime)^{-A_{i,l}}|_{H_{j-1}}$ is a non-zero rational function in $H_{j-1}$. Assume by induction that
$\pi_i\cdot\nolinebreak\prod_{l=1}^{j-1}(\pi_l^\prime)^{-A_{i,l}}$ is defined at some point and non zero in $H_{j-2}\supset H_{j-1}$. Then
\[
\upsilon_{H_{j-1}}\left(\pi_i\cdot\prod_{l=1}^{j-2}(\pi_l^\prime)^{-A_{i,l}}\cdot\left.(\pi_{j-1}^\prime)^{-A_{i,j-1}}\right|_{H_{j-2}}\right)=A_{i,j-1}-A_{i,j-1}\upsilon_{H_{j-1}}(\pi_{j-1}^\prime)=0
\]
and so $\pi_i\cdot\prod_{l=1}^{j-1}(\pi_l^\prime)^{-A_{i,l}}$ is defined at some point and non zero in $H_{j-1}$.
Now for $f\in k(X)^*, k_1=\upsilon_{H_1}(f)=k_1^\prime$ so 
$k_1^\prime=k_1A_{1,1}$.
By induction assume $k_j^\prime=\sum_{l=1}^jk_lA_{l,j}, j=1,...,i$. Then
\[
\pi_{i+1}^{\prime\,k_{i+1}^\prime}\cdots\pi_r^{\prime\, k_r^\prime}u^\prime=
\pi_1^{k_1}\cdots\pi_r^{k_r}u(\pi_1^{\prime\,k_1^\prime}\cdots\pi_i^{\prime\, k_i^\prime})^{-1}=
\]
\[
=\pi_1^{k_1}\cdots\pi_r^{k_r}u(\pi_1^{\prime\,k_1A_{1,1}}\cdots\pi_i^{\prime\, \sum_{l=1}^ik_lA_{l,i}})^{-1}=
\]
\[
=[\pi_1\pi_1^{\prime(-A_{1,1})}\cdots\pi_i^{\prime(-A_{1,i})}]^{k_1}\cdots
[\pi_i\pi_i^{\prime(-A_{i,i})}]^{k_i}
\pi_{i+1}^{k_{i+1}}\cdots\pi_r^{k_r}u
\]
And so
\[
k_{i+1}^\prime=\upsilon_{H_{i+1}}(\left.\pi_{i+1}^{\prime\,k_{i+1}^\prime}\cdots\pi_r^{\prime\, k_r^\prime}u^\prime\right|_{H_i})=
\]
\[
=
k_1\cdot\upsilon_{H_{i+1}}\left(\pi_1\cdot\prod_{l=1}^{i}\left.(\pi_l^\prime)^{-A_{1,l}}\right|_{H_i}\right)+...+
k_i\cdot\upsilon_{H_{i+1}}\left(\pi_i\cdot\prod_{l=i}^{i}\left.(\pi_l^\prime)^{-A_{i,l}}\right|_{H_i}\right)+
\]
\[
+\upsilon_{H_{i+1}}(\pi_{i+1}^{k_{i+1}}|_{H_i})
=\sum_{l=1}^{i+1}k_lA_{l,i+1}
\]
\end{proof}
\subsection{Convex Subgroups}
We will now show that the rank of a valuation
can be interpreted as the number of convex subgroups of the valuation group.
This interpretation will be helpful to us in the sequel.
\begin{defn}
A subgroup $H\subset G$ is \textit{convex} (or \textit{isolated}) if for any $h\in H$ and $g\in G$
\[
\mbox{if }h\ge g\ge 0 \Rightarrow g\in H
\]
\end{defn}
Two convex subgroups of a group $G$ have the property that one is completely contained in the other, and all convex subgroups
$\{G_i\}$ can be arranged as
\[
G\supset G_1\supset...\supset G_m=0
\]
(See \cite[17.1.Thm.4\&Thm.5]{HP})\\
\begin{lem}
The maximal number of (proper) distinct convex subgroups is equal to the rank of $G$.
\end{lem}
\begin{proof}
Let $G$ be of rank $n$ and $\alpha_1>\alpha_2>...>\alpha_n$ be generators, so that
\[
G=\zz\alpha_1\oplus...\oplus\zz\alpha_n
\]
Set 
\[
G_i=\zz\alpha_{i+1}\oplus...\oplus\zz\alpha_n\quad i=1,...,n
\]
If we identify $G$ with $\zz^n$ then the induced ordering on $\zz^n$ is lexicographic.\\
If $h=(0,...,0,h_{i+1},...,h_n)\in G_i, g=(g_1,...,g_n)$ such that 
$$h\ge g\ge0\Rightarrow
0=h_j\ge g_j\ge0, j=1,...,i\Rightarrow g\in G_i$$
and so the $G_i$ are distinct convex subgroups.\\
On the other hand if there is a convex subgroup $H, G_{i-1}\supset H \supsetneq G_i$, then there is a $h\in H$,
\[
h=(0,...,0,h_i,h_{i+1},...,h_n)\quad h_i\ne 0
\]
\[
\Rightarrow h\ge(0,...,1_i,0,...,0)\ge 0\Rightarrow (0,...,1_i,0,...,0)\in H \Rightarrow H=G_{i-1}
\]
Therefore there can not be more than $n$ convex subgroups.
\end{proof}

\section{Main Theorems}
In this section we will use valuation theory to prove lower bounds on $\trdeg K$ of any field $K/k$ and as a corollary, for $\dim f(X)$, where $f:X\dashrightarrow \mathbb{A}^n$ is a given rational map of $k-$varieties. This will be a source for many interesting applications.\\
Theorem \ref{mthm} is the general statement of the inequality. Another proof for it can be found in \cite[12.4. Theorem II]{HP}. Theorem \ref{mthm1} is essentially the same statement but for the case of a function field of a variety with a valuation given by a chain of divisors as in \ref{val}. In this case one can explicitely construct algebraically independent elements.
\begin{thm}\label{mthm}
Let $K$ be a field extension of $k$, $\upsilon:K^*\rightarrow G$ a valuation (that is trivial on $k$). Then $\trdeg K\geq\rank \upsilon$.
\end{thm}
\begin{proof}
Suppose $\rank \upsilon=s$ i.e. $G=\upsilon(K)$ is of rank $s$. 
Let $\{G_i\}$ be the convex subgroups of $G$, so that
\[
G=G_0\supset G_1\supset...\supset G_s=0.
\]
Let $R$ be the valuation ring 
\[
R=\{f\in K\mid \upsilon(f)\ge 0\}\cup\{0\}
\]
and let 
\[
\mathfrak{m}_i=\{f\in R\mid\upsilon(f)\notin G_i\}
\]
We show that $\mathfrak{m}_i$ is a prime ideal.
If $f\in\mathfrak{m}_i, g\in R$ then $\upsilon(fg)=\upsilon(f)+\upsilon(g)\ge\upsilon(f)\ge 0$ Now since $G_i$ is convex and $\upsilon(f)\notin G_i$, it follows that $\upsilon(fg)\notin G_i$ and so $fg\in\mathfrak{m}_i$. If $f,g\in\mathfrak{m}_i\Rightarrow \upsilon(f+g)\ge\min\{\upsilon(f),\upsilon(g)\}\ge 0$ and hence $f+g\in\mathfrak{m}_i$. Thus $\mathfrak{m}_i$ is an ideal in $R$.\\
If $f,g\notin\mathfrak{m}_i\Rightarrow \upsilon(f), \upsilon(g)\in G_i\Rightarrow \upsilon(fg)\in G_i, fg$ cannot be in $\mathfrak{m}_i$ and so $\mathfrak{m}_i$ is prime.\\
Now $\mathfrak{m}_i\subsetneq\mathfrak{m}_{i+1}$ (note that $\upsilon$ is surjective) and we have a chain of prime ideals
\[
\{0\}=\mathfrak{m}_0\subsetneq\mathfrak{m}_1\subsetneq...\subsetneq\mathfrak{m}_{s}
\]
which means $\dim R\ge s$\\
The quotient field $K$ of $R$ must have at least $\dim R$ algebraically independent elements, so
\[
\trdeg K\ge\dim R\ge s
\]
\end{proof}
\begin{cor}\label{mcor}
Let $f=(f_1,...,f_n):X\dashrightarrow Y\subset \mathbb{A}^n$ be a dominant rational map of $k-$varieties and $\upsilon:k(X)^*\rightarrow\zz^r$ a valuation. Denote $\upsilon(f)=$\footnotesize$\left(\begin{array}{c}\upsilon(f_1)\\ : \\ \upsilon(f_n)\end{array}\right)$\normalsize$\in M_{n,r}(\zz)$. Then $\dim Y\ge \rank \upsilon(f)$.
\end{cor}
\begin{proof}
Look at the induced valuation on $k(Y)^*\subset k(X)^*$. $\dim Y =\dim f(X)=\trdeg f(X)\ge \rank \upsilon(f(X))=\rank \upsilon(f)$.
\end{proof}
If the field $K$ is the function field of a variety and the valuation given by divisors, we can restate Theorem \ref{mthm} and give a constructive proof:
\begin{thm}\label{mthm1}
Let $X$ be a normal irreducible variety, $\upsilon:k(X)^*\rightarrow\zz^n$ a valuation constructed as in \ref{val} with local parameters $\pi_1,...,\pi_n$ and $f_1,..,f_r \in k(X)^*$ be rational functions. Denote V=\footnotesize$\left(\begin{array}{c}\upsilon(f_1)\\ : \\ \upsilon(f_r)\end{array}\right)$\normalsize$\in M_{r,n}(\zz)$. \\
Then $\trdeg k(f_1,...,f_r)\geq\rank V$.
\end{thm}
\begin{proof}
Suppose $\rank V=s$ ($s\leq r)$. For notational simplicity assume $V=\left(\begin{array}{c|c}\tilde{V}&*\\\hline *&*\end{array}\right)$ where $\tilde{V}\in\zz^{s\times s}$ and det \nolinebreak$\tilde{V}\ne0$. Take its $\qq-$inverse and multiply it by a suitable $\lambda \in \nn$ such that $\lambda\tilde{V}^{-1}\in\zz^{s\times s}$ and $ \lambda\tilde{V}^{-1}\tilde{V}=\lambda I$. Set $\Lambda=\left(\begin{array}{c|c}\lambda\tilde{V}^{-1}&0\end{array}\right)\in
\zz^{s\times r}$.
Consider the functions 
\[
g_j=f_1^{\Lambda_{j1}}\hdots f_r^{\Lambda_{jr}}\in k^*(X), \quad j=1,..,s.
\]
\[
\upsilon(g_j)=\upsilon(f_1^{\Lambda_{j1}}\hdots f_r^{\Lambda_{jr}})=\sum_{l=1}^r\Lambda_{jl}\upsilon(f_l)=
\]
\[
=(\Lambda_{j1},...,\Lambda_{jr})\cdot V=(0,..,0,\lambda,0,..,0,*,...,*)
\]
where the lambda is at the $j$th position.
Thus $g_j={\pi_j}^\lambda \cdot u$ for some unit $u|_{H_{s-1}} \in \mathcal{O}_{H_{s-1},H_s}$.\\
Note that $g_1|_{H_1}\equiv0$ since $\upsilon_{H_1}(g_1)=\lambda>0$ and $g_2,..,g_s$ are units in $\mathcal{O}_{X,H_1}$ hence non-zero rational functions in $H_1$.\\
Suppose now that $g_1,..,g_s$ are algebraically dependent over $k$, i.e. there exists an irreducible polynomial $P(t_1,..,t_s)\in k[t_1,..,t_s], P\ne0$ and $$P(g_1,..,g_s)=0\quad \mbox{on }X$$ 
Then
$$0=P(g_1|_{H_1},..,g_s|_{H_1})=P(0,g_2|_{H_1},..,g_s|_{H_1})$$ 
If $P(0,t_2,..,t_s)\not\equiv 0$ then $g_2|_{H_1},..,g_s|_{H_1}$ are algebraically dependent, otherwise by irreducibility $P=ct_1, \quad c \in k^*$ and therefore $g_1\equiv0$ on $X=H_0$ which is what we excluded by hypothesis.
Inductively we can conclude that $g_j|_{H_{j-1}}\equiv0$ for some $j\le s$ but by construction $g_j|_{H_{j-1}}$ is a non-zero rational function. So we must have $s$ algebraically independent elements $g_1,..,g_s$ in $k(f_1,..,f_r)$, thus $\dim f(X)= \trdeg(f_1,..,f_r)\ge s.$
\end{proof}

Theorem \ref{cthm} will be a short digression. We try to give a converse to Corollary \ref{mcor}: For a variety of dimension $n$ does there exist a valuation of rank $n$ for its function field?
\begin{thm}\label{cthm}
Let $K$ be the function field of a variety of dimension $n$ over $k$ and $f_1,...,f_r$
$(r\le n)$ be algebraically independent elements of $K$. Then there exists a variety $X$ with function field $k(X)=K$ and a valuation $\upsilon: k(X)^*\rightarrow\zz^n$ such that the induced valuation $\upsilon: k(f_1,...,f_r)^*\rightarrow\zz^n$ is of rank $r$.
\end{thm}
\begin{proof}
First we can assume $r=n$: If we extend the functions to a transcendence base $f_1,...,f_n$ of $K$ and have a valuation $\upsilon:k(f_1,...,f_n)\rightarrow \zz^n$ of rank $n$, then by Theorem \ref{mthm} the induced valuation on $k(f_1,....,f_r)$ is of rank $r$.
Let $X$ be an affine irreducible variety with $k(X)=K$.\\
Let $K^\prime=k(f_1,...,f_n)$. $K$ is a finitely generated algebraic (i.e. finite) extension of $K^\prime$.\\
Define a valuation $\upsilon^\prime:K^\prime\rightarrow G^\prime=\zz^n$ by setting 
\[
\upsilon^\prime\left(\sum_{i}c_i f^{e_i}\right)=\min_{c_i\ne 0}\{e_i=(e_{i1},...,e_{in})\}\in\zz^n
\]
where $f^{e_i}=f_1^{e_{i1}}\cdots f_n^{e_{in}}$, only finitely many of the $c_i\in k$ are non-zero and $\zz^n$ is ordered lexicographically.\\
Obviously $\upsilon^\prime(f_1),...,\upsilon^\prime(f_n)$ generate $G^\prime=\zz^n$ and $\upsilon^\prime$ is of rank $n$.\\
Then $\upsilon^\prime$ can be extended to a valuation
\[
\upsilon:K\rightarrow G\supseteq G^\prime
\]
which is also of rank $n$ since $K/K^\prime$ is finite (\cite[XII.2,Thm.7\&Thm.10]{HP}).

\end{proof}

\begin{rem}
The question whether there exists a valuation of rank $r$ that is constructed through a chain of divisors as in \ref{val} is more subtle. The main difficulty is to make sure that the divisors are normal. 
\end{rem}

\subsection{Lemma on the Rank of a Matrix}
For later use we write down here two simple Lemmas on the rank of a matrix.
\begin{lem}\label{rk1}
Let $A$ be an integer matrix and $m$ an integer, $m\ne0$.\\
Then $\rank A \ge \rank (A \mod m)$.
\end{lem}
\begin{proof}
Let $\rank A= r.$ Then every $r+1$ minor of $A$ is zero and so it is$\mod m$. Thus $\rank (A \mod m)\le r$.
\end{proof}
\begin{lem}\label{rk2}
Let $A,B$ be matrices of the same size.
Then \\
i)\quad $\rank (A+B) \le \rank A +\rank B$\\
ii)\quad $\rank (A+B) \ge |\rank A -\rank B|$
\end{lem}
\begin{proof}
i)\quad 
Clearly $\rm{Im} (A+B)\subset \rm{Im} (A) +\rm{Im} (B)$. It follows that $\rank (A+B) \le \rank A +\rank B$\\
ii)\quad Assume that $\rank A\ge \rank B$. ii) then follows from i) by replacing $A$ by $A+B$ and $B$ by $-B$.
\end{proof}

\section{Orbits of Torus Actions}\label{ota}
Now we turn our attention towards essential dimension.
We will look at the action of an $m-$dimensional torus $T$ on $K^n$ where $K$  is a field extension of $k$. The orbits of this action are the objects whose essential dimension we are interested in. All information of the action is stored in a matrix $E$ and the functor of orbits will be denoted by $\mathcal{F}_E$. We will recall the definition of $K-$points and the generic point $x$ of $K^n$. We will compute the essential dimension of the orbit of $x$ and thus get the essential dimension of $\mathcal{F}_E$ by showing that $\ed[x]=\ed(\mathcal{F}_E)$.
\subsection{Elliptic Curves}\label{elcex}
As a motivating example consider the set of elliptic curves over $K$
\[
\{(x,y)\in K^2 \mid y^2=x^3+ax+b; a,b\in K\}
\]
Two curves $\{(x,y)\},\{(x^\prime,y^\prime)\}$ are isomorphic if there exists a $\lambda \in K^*$ such that $x=\lambda^2x^\prime, y=\lambda^3y^\prime$ so that
$\lambda^6{y^\prime}^2=\lambda^6{x^\prime}^3+a\lambda^2x^\prime+b$ or
$${y^\prime}^2={x^\prime}^3+a\lambda^{-4}x^\prime+b\lambda^{-6}$$
 which means $a^\prime=a\lambda^{-4}, b^\prime=b\lambda^{-6}$ (See for example \cite[III.1]{Si}).
\\
Thus we can identify elliptic curves over $K$ with pairs $(a,b)\in K\times K$ where 
$$(a,b)\sim(a^\prime,b^\prime)\iff \exists \lambda \in K^*, a^\prime=a\lambda^4, b^\prime=b\lambda^6.$$
Let 
\[
\mathcal{F}_{El}(K)=(K\times K)/\sim
\]
be the functor that assigns to every field $K$ the elliptic curves defined in $K$.
We will see later, that $\mbox{ed}(\mathcal{F}_{El})=2$.

\subsection{The Functor $\mathcal{F}_E$}\label{tact}
The example of elliptic curves leads us to a more general concept.

For fixed $n,m\in\nn$ consider 
\[
\mathcal{F}_E(K)=K^n/\sim\quad\mbox{where}
\]
\[
(a_1,...,a_n)\sim({a^\prime}_1,...,{a^\prime}_n)\iff \exists \lambda_1,...,\lambda_m \in K^* \mbox{ such that }
\]
\[
\begin{array}{c}a^\prime_1=a_1\lambda_1^{e_{11}}\cdots\lambda_m^{e_{1m}}\\
a^\prime_2=a_2\lambda_1^{e_{21}}\cdots\lambda_m^{e_{2m}}\\
\vdots\\
a^\prime_n=a_n\lambda_1^{e_{n1}}\cdots\lambda_m^{e_{nm}}
\end{array}
\]
The $e_{ij}\in \zz$ are fixed exponents and we write them in a matrix
\[
E=\left(\begin{array}{ccc}e_{11}&\cdots&e_{1m}\\
\\\vdots&&\vdots\\\\
e_{n1}&\cdots&e_{nm}\end{array}\right)\in M_{n,m}(\zz)
\]
Clearly $\ed(\mathcal{F}_E)$ only depends on $E$.

\subsection{$K-$points and the generic point}
The notion of $K-$points gives a more general way of looking at $\mathcal{F}_E$:
\begin{defn}
Let $X$ be a variety over $k$ and $K/k$ a field extension. A \textit{$K-$point} $\xi$ is a rational map $\xi:Y\dashrightarrow X$ where $Y$ is a variety with $k(Y)=K$. The set of $K-$points is denoted by $X(K)$.
\end{defn}
Note: $Y$ is only defined up to to birational isomorphism.
\begin{exs}
1) $K=k$. Then $X(k)=\{\ast\rightarrow X\}=\{\mbox{points of }X\}$\\
2) $K=k(X)$. Then $X(k(X))=\{\phi:X\rightarrow X\}=$ rational maps. $id_X\in X(k(X))$ is called the \textit{generic point} of $X$.\\
3) $X=k^n$. A rational map $\phi:Y\dashrightarrow X$ with $k(Y)=K$ can be viewed as an element in $K^n$, $k^n(K)=K^n$.\\
4) $X=(k^*)^n$. Then $(k^*)^n(K)=(K^*)^n$.\\
\end{exs}
Now let $X=k^n$ and $T=(\mathbf{G}_m)^m=(k^*)^m$. The torus $T(K)=(K^*)^m$ acts on $X(K)=K^n$ through characters
\[
\chi_i(t)=\chi_i(t_1,...,t_m)=t_1^{e_{i1}}\cdots t_m^{e_{im}}\quad i=1,...,n
\]
\[
t\cdot x=(x_1\chi_1(t),...,x_n\chi_n(t))
\]
View now the functor $\mathcal{F}_E$ as 
\[
\mathcal{F}_E(K)=X(K)/\sim
\]
\[
a\sim b\iff a=t\cdot b\;\mbox{ for a }t\in T(K)
\]
\begin{defn}
Let $x=(x_1,...,x_n)$ be the generic point of $k^n$. The \textit{generic point of }$\mathcal{F}_E$ is the orbit $[x]\in\mathcal{F}_E(k(x_1,...,x_n))$.
\end{defn}
Our final goal is to find $\ed(\mathcal{F}_E)$. We begin with a simple bound on $\ed(\mathcal{F}_E)$:
\begin{lem}
$\ed(\mathcal{F}_E)\le n$.
\end{lem}
\begin{proof}
Let $[a]=[(a_1,...,a_n)]\in \mathcal{F}_E(K)$ be any equivalence class (orbit). Then $a\in k(a_1,...,a_n)^n\subset X(K)$ and so
$\ed([a])\le \trdeg k(a_1,...,a_n)\le n$. It follows that $\ed(\mathcal{F}_E)\le n$.
\end{proof}
So we can restrict our attention to $k(a_1,...,a_n)$ and, searching for the minimal transcendence degree for an arbitrary equivalence class, replace the $a_i$ by variables $x_i$ which motivates the use of the generic point $[x]$. Define
\[
\begin{array}{c}
f_1=x_1t_1^{e_{11}}t_2^{e_{12}}\cdots t_m^{e_{1m}}=x_1\chi_1(t)\\
f_2=x_2t_1^{e_{21}}t_2^{e_{22}}\cdots t_m^{e_{2m}}=x_2\chi_2(t)\\
\vdots \\
f_n=x_nt_1^{e_{n1}}t_2^{e_{n2}}\cdots t_m^{e_{nm}}=x_n\chi_n(t)\\
\end{array}
\]
So that
\begin{equation}\label{edgp}
\ed[x]=\min \{\trdeg k(f_1,...,f_n)\mid t=(t_1,...,t_m)\in T(k(x_1,...,x_n))\}.
\end{equation}
That suggests using Theorem \ref{mthm} to find the essential dimension of the generic point.
To construct a valuation on the function field of $X=k^n$ we may choose as our set of divisors  
$H_1=\{x_1=0\},H_2=\{x_1=x_2=0\},...,H_n=\{x_1=x_2=...=x_n=0\}$ as in example \ref{mex}.
We have 
\[
\upsilon(f_i)=(0,..,1_i,..,0)+ e_{i1}\upsilon(t_1)+...+e_{im}\upsilon(t_m)\;\quad i=1,...,n
\]
Setting
\[
U=\left(\begin{array}{ccccc}u_{11}&&\cdots&&u_{1n}\\
\vdots&&&&\vdots\\
u_{m1}&&\cdots&&u_{mn}\end{array}\right):=
\left(\begin{array}{c}\upsilon(t_1)\\ \vdots\\ \upsilon(t_m)\end{array}\right)\in M_{m,n}(\zz)
\]
we get
\begin{equation}
\left(\begin{array}{c}\upsilon(f_1)\\ \vdots \\\upsilon(f_n)\end{array}\right)=
I+EU
\end{equation}
\begin{lem}\label{edl}
$\ed[x]\ge\rank(I+EU)\ge n-\rank\;E$.
\end{lem}
\begin{proof}
Clearly if $\rank E=r$ then $\rank EU\le r$ and the second inequality follows from Lemma \ref{rk2}. The first inequality is immediate from (\ref{edgp}) and Theorem \ref{mthm1}.
\end{proof}

The rank of $E$ alone is not all that determines $\ed[x]$. To see this, look for example at $E=(2,2,...,2)$ and  
\[
\begin{array}{c}f_1=x_1t^2\\ \vdots\\f_n=x_nt^2
\end{array}
\]
Then the valuation matrix $\upsilon((f_1,...,f_n)^T)\mod 2=I$ and so there is no way to choose a $t$ such that $\trdeg k(f_1,...,f_n)=n-1$. In fact it will turn out that for $m=1$ we need $\gcd\{e_{11},...,e_{n1}\}=1$ to get $\ed[x]=n-1$.

\subsection{Essential Dimension of $\mathcal{F}_E$ and the generic point}
Now look at the matrix $E\in\zz^{n\times m}$. By elementary row- and column operations (over $\zz$) it can be transformed into Smith normal form 
\[
\left(\begin{array}{ccc}d_1&&\\
&\ddots&\\
&&d_r\\
&&\end{array}\right)\in M_{n,m}(\zz)\quad (r=\rank E)
\]
such that the \textit{elementary divisors} $d_i$ divide each other, $d_1\mid d_2\mid...\mid d_r$.\\

\begin{thm}\label{edgen}
Let $E$ be the matrix of exponents in $f_1,...,f_n$ and $d_1,...,d_r$ be the elementary divisors of its Smith normal form. If $l$ is the number of ones among $\{d_1,...,d_r\}$, then $\ed[x]=n-l$.
\end{thm}
\begin{proof}
First we note that we can replace $f_i$ by $f_i^{-1}$ or by $f_if_j^\lambda$ ($i\ne j, \lambda\in\zz$) or we can switch $f_i$ and $f_j$ without changing the transcendence degree of $k(f_1,...,f_n)$. These changes translate into row operations on $\upsilon((f_1,...,f_n)^T)$: Multiply row $i$ by $-1$, add $\lambda\cdot$ row $j$ to row $i$, switch row $i$ and row $j$.\\
Row operations are achieved by multiplication from the left by a matrix $P\in \mbox{GL}_n(\zz)$.
So the valuation matrix becomes
\[
\left(\begin{array}{c}\upsilon(f_1)\\ \vdots \\\upsilon(f_n)\end{array}\right)=
PI+PEU
\]
We can also do similar changes in the $x_i$ which affects the valuation matrix by multiplication of a $\zz-$invertible matrix $R$ from the right:
\[
\left(\begin{array}{c}\upsilon(f_1)\\ \vdots \\\upsilon(f_n)\end{array}\right)=
PIR+PEUR
\]
Replacing the matrix of variables $U$ by $Q^{-1}UR$ we get
\[
\left(\begin{array}{c}\upsilon(f_1)\\ \vdots \\\upsilon(f_n)\end{array}\right)=
PIR+PEQU
\]
or in other words, we can assume we have the system
\[
\left(\begin{array}{c}\upsilon(f_1)\\ \vdots \\\upsilon(f_n)\end{array}\right)=
I+EU
\]
where $E$ is in Smith normal form with elementary divisors $d_1,...,d_r$ (take $R=P^{-1}\in \mbox{GL}_n(\zz)$. Thus
\begin{equation}\label{rac}
f_i=
\left\{\begin{array}{ll}
x_it_i^{d_i}&i\le r\\x_i&i>r
\end{array}\right.
\end{equation}
Note: Row operations must be done simultaneously on both summands $I, EU$, whereas column operations can be done independently.\\
Now let $l$ be the number of ones among $d_1,...,d_r$, which means $d_1=...=d_l=1, d_{l+1}\ne1$. We can choose
$t_i=x_i^{-1}$ and so $f_i=1$ for $i=1,...,l$. Thus $\ed[x]\le n-l$.\\
On the other hand, $d_{l+1}\equiv...\equiv d_r\equiv0 \mod d_{l+1}$ and $E \mod d_{l+1}$ is of rank $l$. Therefore, $\rank[(I+EU)\mod d_{l+1}] \ge n-l$ and also $\rank(I+EU)\ge n-l$ (Lemma \ref{rk1}\&\ref{rk2}).
Finally, by Lemma \ref{edl} $\ed[x]\ge \rank \upsilon((f_1,...,f_n)^T)\ge n-l$.
\end{proof}
Our next step is to show that the essential dimension of any $K-$point can not exceed the essential dimension of the generic point and hence $\ed(\mathcal{F}_E)=\ed[x]$. For this we use the specific structure of the given action. In general, for an arbitrary action of an algebraic group, it is not known if the essential dimension of the generic point is equal to the essential dimension of the functor of orbits (\cite[Remark 14.3]{BR3}). First we need a short lemma.
\begin{lem}\label{emat}
Let $\tilde{E}\in M_{\tilde{n},m}(\zz)$ and $E=\left(\begin{array}{c}\tilde{E}\\\hline \ast\end{array}\right)\in M_{n,m}(\zz)$ an extension of $\tilde{E}$. Let $[\tilde{x}]$ and $[x]$ be the generic points of $\mathcal{F}_{\tilde{E}}$ and $\mathcal{F}_E$ respectively. Then $\ed[x]\ge\ed[\tilde{x}]$.
\end{lem}
\begin{proof}
Let $l,\tilde{l}$ be the number of ones among the elementary divisors of $E$ and $\tilde{E}$ respectively, so $\tilde{l}=\rank (\tilde{E}\mod p)$ for a suitable $p>1$. But then $l\le\rank (E \mod p)\le n-\tilde{n}+\tilde{l}$ and
\[
\ed[x]=n-l\ge n-(n-\tilde{n}+\tilde{l})=\ed[\tilde{x}]
\]
\end{proof}
\begin{thm}\label{edF}
Let $E$ be a matrix with elementary divisors $d_1,...,d_r$ and $l$ the number of ones among $\{d_1,...,d_r\}$. Let $[x]$ be the generic point of $\mathcal{F}_E$. Then $\ed(\mathcal{F}_E)=\ed[x]=n-l$.
\end{thm}
\begin{proof}
We only need to proof the first equality. Clearly $\ed(\mathcal{F}_E)\ge\ed[x]$ by definition of $\ed(\mathcal{F}_E)$.
Let $[a]$ be an equivalence class of any $K-$point $a\in X(K)=K^n$. We would like to have $\ed[a]\le\ed[x]$. If $a_i\ne 0$ for all $i=1,...,n$ we can assume that our system is as in (\ref{rac}) and simply take $t_i=a_i^{-1}, i=1...,l$ to get $\ed[a]\le\trdeg k(a_{l+1},...,a_n)\le\ed[x]$. However, if one of the $a_i$ is zero, this is not working. Without loss of generality assume that $a_1=...=a_{\tilde{n}}\ne 0$ for some $\tilde{n}<n$ and $a_i=0, i>\tilde{n}$. Let $\tilde{E}$ be the first $\tilde{n}$ rows of $E$ and $\tilde{a}=(a_1,...,a_{\tilde{n}})$. Then by the previous argument and Lemma \ref{emat},
\[
\ed[a]=\ed[\tilde{a}]\le\ed[\tilde{x}]\le \ed[x]
\]
as desired.
\end{proof}
For the case with only one function $t$ ($m=1$), we have
\begin{cor}
Let $E=(e_1,...,e_n)$ be a tuple of exponents. Then 
\[
\ed(\mathcal{F}_E)=\left\{\begin{array}{ll}n-1& \gcd\{e_1,...,e_n\}=1\\
n&else
\end{array}\right.
\]
\end{cor}
For the elliptic curves in the introductory example \ref{elcex} we get
\begin{cor}\label{elc}
$\ed(\mathcal{F}_{El})=2$.
\end{cor}
\begin{proof}
We have $E=(4,6)$ with $4,6$ not relatively prime. So $\ed(\mathcal{F}_{El})=2$.
\end{proof}

\subsection{Example: Homogeneous forms in a rigid system}\label{hf}
For a fixed field $K/k$ let $V$ be a $m-$dimensional vector space over $K$, $L_1,...,L_m$ be lines in general
position, (i.e. they span $V$), and $f:V\rightarrow K$ a homogeneous form of degree $d$. 
Consider the set of tuples $\{(V,f,L_1,...,L_m)\}$. Two tuples $(V,f,L_1,...,L_m)$,
$(V^\prime,f^\prime,L_1^\prime,...,L_m^\prime)$ are equivalent $\iff$ there exists an isomorphism
\[
\phi:V\rightarrow V^\prime\quad\mbox{such that}
\]
\[
L_i\mapsto L_i^\prime\quad\mbox{i=1,...,m}
\]
\[
f=f^\prime\phi
\]
Let $\mathcal{F}_{m,d}(K)$ be the set of all such tuples over the field $K$ modulo equivalence.\\
We call $\mathcal{F}_{m,d}(K)$ homogeneous forms in a rigid system.\\
Choose a basis $(v_1,...,v_m)$ with $v_i\in L_i$ then $f$ becomes a homogeneous polynomial of degree $d$ in the
coordinates:
\[
v=\sum_{i=1}^mx_iv_i\in V
\]
\[
f(v)=f(x)=\sum_{d_1+...+d_m=d}a_{d_1,...,d_m}x_1^{d_1}\cdots x_m^{d_m}
\]
The coefficients $a_{d_1,...,d_m}$ determine the form $f$ for a given basis (set of lines) and
we can identify $(V,f,L_1,...,L_m)$ by the tuple of these coefficients. $d_1,...,d_m$ partition $d$ into $m$ pieces and there are $N={{m+d-1}\choose{d}}$ such partitions. So if we choose an order, $a=(a_{d_1,...,d_m})_{d_1+...+d_m=d}\in K^N$.
\[
a\sim a^\prime\iff
\]
\[
\exists\;\phi:V\rightarrow
V^\prime, v_i\mapsto \lambda_iv_i^\prime, \lambda_i\in K^*, i=1,...,m, f=f^\prime\phi\iff
\]
\[
f(x_1,...,x_m)=f^\prime(\lambda_1x_1,...,\lambda_mx_m)
\iff
\]
\[
a_{d_1,...,d_m}=a_{d_1,...,d_m}^\prime\lambda_1^{d_1}\cdots \lambda_m^{d_m}\quad d_1+...+d_m=d
\]
Thus we are in the familiar setting $\mathcal{F}_{m,d}(K)=K^N/\sim$ where the equivalence
relation is expressed in the exponent matrix
\[
E=\left(\begin{array}{cccc}d&&&\\
d-1&1&&\\
&\ddots&&\\
\\
\ddots&&&\\
&&&d
\end{array}\right)\in M_{N,m}(\zz)
\]
Every row of $E$ is a partition of $d$.\\
Now we compute the Smith normal form of $E$. In $E$ the following rows occur:
\[
\left(\begin{array}{ccccc}1&\ast&&\cdots&\ast\\
0&1&\ast&\cdots&\ast\\
&&\ddots&&\\
0&\cdots&0&1&d-1
\end{array}\right)
\]
Therefore, the first $m-1$ elementary divisors of the Smith normal form are $1$. All rows of $E$ sum up to $d$.
After applying row-operations the rows sum up to integer multiples of $d$. Since also $(0,...,0,d)$ occurs, the
row reduced matrix must be
\[
\left(\begin{array}{ccccc}1&\ast&&\cdots&\ast\\
&1&\ast&\cdots&\ast\\
&&\ddots&&\\
&&&1&d-1\\
&&&&d\\
\\
\\
\\
\end{array}\right)
\]
and hence the elementary divisors of the Smith normal form are $1$ ($m-1$ times) and $d$.\\
It follows from Theorem \ref{edF} that
\[
\ed(\mathcal{F}_{m,d})=N-\mbox{\#ones in Smith normal form}
\]
\[=
\left\{\begin{array}{ll}{{m+d-1}\choose d}-m+1&d>1\\0&d=1\end{array}\right.
\]
For example for $d=2$ the essential dimension of quadratic forms (in a rigid system) in dimension $m$ are:
\[
\begin{array}{c|c}m&\ed(\mathcal{F}_{m,2})\\
\hline
1&1\\
2&2\\
3&4\\
4&7\\
5&11\\
6&16\\
:&:
\end{array}
\]
\begin{rem}
For more information on the essential dimension of homogeneous forms (in a non rigid system), we refer to \cite{BR3}.
\end{rem}

\section{Torus Action on Projective Space}\label{ps}
Let $K$ be a field extension of $k$. So far we only looked at the action of the torus $T(K)$ on the affine space $\mathbb{A}^n(K)$. We will now consider the action on projective space $\mathbb{P}^n(K)$. Any action of the torus $T$ on $k^{n+1}$ which is determined by a matrix $E\in M_{n+1,m}(\zz)$ as in \ref{tact} defines an action on $\mathbb{P}^n(K)$. Let
\[
\mathcal{PF}_E(K)=\mathbb{P}^n(K)/\sim
\]
where 
\[
a\sim b\mbox{ if } a=t\cdot b\;\mbox{ for a }t\in T(K)
\]
Here we have the additional requirement that $a\sim
ca$ for any $c\in K^*$, since they are equal in $\mathbb{P}^n(K)$. But then we can just add an extra function $t_{m+1}$ in every coordinate and return to the affine case with the extended matrix
\[
\tilde{E}=\left(\begin{array}{ccc|c}&&&1\\
&E&&\vdots\\
&&&1
\end{array}\right)\in M_{n+1,m+1}(\zz)
\]
So that 
\[
\ed(\mathcal{PF}_E)=\ed(\mathcal{F}_{\tilde{E}})
\]
\begin{prop}\label{F-PF}
$\ed(\mathcal{F}_E)\ge\ed(\mathcal{PF}_E)\ge\ed(\mathcal{F}_E)-1$
\end{prop}
\begin{proof}
Let $l$ and $\tilde{l}$ be the number of ones of the elementary divisors of $E$ and $\tilde{E}$ respectively, so that $\ed(\mathcal{F}_E)=n+1-l$ and $\ed(\mathcal{PF}_E)=n+1-\tilde{l}$. Obviously $l\le\tilde{l}$ and the first inequality follows.
We have $l=\rank(E \mod d_{l+1})$ where $d_{l+1}$ is the $(l+1)$th elementary divisor. Then $\tilde{l}\le\rank(\tilde{E} \mod d_{l+1})\le l+1$ and the second inequality follows.
\end{proof}
To determine whether $\ed(\mathcal{PF}_E)$ is less than $\ed(\mathcal{F}_E)$ or not, one simply has to find the Smith normal form of $\tilde{E}$. There is no better formula than this. To illustrate it consider the two matrices
\[
\begin{array}{lcr}
E_1=\left(\begin{array}{c}2\\3\end{array}\right)&&E_2=\left(\begin{array}{c}-2\\3\end{array}\right)
\end{array}
\]
While in both cases $\ed(\mathcal{F}_{E_1})=\ed(\mathcal{F}_{E_2})=1$, the extended matrices are transformed to
\[
\begin{array}{lcr}
\tilde{E}_1\rightsquigarrow\left(\begin{array}{cc}1&0\\0&1\end{array}\right)&&\tilde{E}_2\rightsquigarrow\left(\begin{array}{cc}1&0\\0&5\end{array}\right)
\end{array}
\]
and so 
\[
\begin{array}{lcr}
\ed(\mathcal{PF}_{E_1})=0&&\ed(\mathcal{PF}_{E_1})=1
\end{array}
\]

\subsection{Example: Hypersurfaces in Projective Space}
This example is the analogue of the homogeneous forms on projective space.
Consider the functor
\[
\mathcal{PF}_{m,d}(K)=\{(V,H,L_1,...,L_m)\}
\]
where $V$ is a $m$ dimensional $K-$ vector space, $L_1,...,L_m$ lines in general position and $H$ is a hypersurface of degree $d$ in the
projective space $\pp(V)$. $H$ is defined through a homogeneous form $f$ of degree $d$ with $f\sim
cf$ for any $c\in K^*$. Identifying $\mathcal{PF}_{m,d}(K)$ with
${m+d-1}\choose{d}$-tuples $a=(a_{d_1,...,d_m})_{d_1+...+d_m=d}$ we get
\[
a\sim a^\prime\iff
\]
\[
\exists\;\lambda_1,...,\lambda_{m+1}\in K^*,\;f(x_1,...,x_m)=\lambda_{m+1}f^\prime(\lambda_1x_1,...,\lambda_mx_m)
\iff
\]
\[
a_{d_1,...,d_m}=a_{d_1,...,d_m}^\prime\lambda_1^{d_1}\cdots \lambda_m^{d_m}\lambda_{m+1}\quad d_1+...+d_m=d
\]
The exponent matrix is then
\[
E=\left(\begin{array}{ccccc}d&&&&1\\
d-1&1&&&1\\
&\ddots&&&\vdots\\
\\
\ddots&&&&\\
&&&d&1
\end{array}\right)\in M_{N,m+1}(\zz)\quad N={{m+d-1}\choose d}
\]
as in the case of homogeneous forms, it reduces to
\[
\left(\begin{array}{cccccc}1&&&&&\\
&1&&&&\\
&&\ddots&&&\\
&&&1&&\\
&&&&d&1\\
&&&&&\ast\\
&&&&&\vdots\\
&&&&&\ast
\end{array}\right)
\mbox{ and then to }
\left(\begin{array}{ccccc}1&&&&\\
&1&&&\\
&&\ddots&&\\
&&&1&\\
&&&&0\\
\\
\\ \\
\end{array}\right)
\]
Hence we get
\[
\ed(\mathcal{PF}_{m,d})={{m+d-1}\choose d}-m
\]

\section{Finite Abelian Groups}\label{FAG}
The valuation theoretic approach and the result of Theorem \ref{mthm} can be used to calculate the essential dimension of other algebraic objects; with its help we show in this section that the essential dimension of finite abelian groups is equal to its rank.

\subsection{Essential Dimension of Algebraic Groups}
Let $K$ be a field extension of $k$, $G(K)$ an algebraic group and the functor $\mathcal{F}_G$ defined by
\[
\mathcal{F}_G(K)=H^1(K,G)
\]
(the Galois-cohomology set, see \cite{BF} for details).
The essential dimension of the group  is defined to be 
\[
\ed(G):=\ed(\mathcal{F}_G)
\]
To apply our results from valuation theory, we recall the original definition of essential dimension of finite groups given in \cite{BR1}.
Let $G$ be a finite group and $X$ a variety over $k$ on which $G$ acts faithfully. 
\begin{defn} A \textit{compression} of $X$ is a faithfull $G-$variety $Y$ and a dominant rational map
\[
f:X\dashrightarrow Y
\]
which is $G-$equivariant.
\end{defn}
The essential dimension of $X$ measures how much $X$ can be compressed:
\begin{defn}
The \textit{essential dimension} of $X$, denoted $\ed(X)$ is defined to be
\[
\ed(X)=\min \dim(Y)
\]
where $Y$ ranges over all compressions of $X$.
\end{defn}
Let $f:X\dashrightarrow Y$ be a compression of $X$ such that $\ed(X)=\dim(Y)$. We can assume that $Y\subset\mathbb{A}^r\;(r\le\dim X)$ is affine and $f=(f_1,...,f_r)$. If we have a valuation $\upsilon:X\rightarrow \zz^n$ we are in the situation of Corollary \ref{mcor}. and get a lower bound on $\ed(X)$:
\begin{prop}\label{pfg}
Let $X$ be a faithfull $G-$variety, $f:X\dashrightarrow Y$ a compression and $\upsilon:k(X)^*\rightarrow \zz^n$ a valuation. Then
\[
\ed(X)\ge \rank \upsilon(f)
\]
\end{prop}
If we consider only linear $G-$varieties (i.e. representations of $G$), then it was proved in \cite[Theorem 3.1]{BR1} that the essential dimension only depends on the group and not on the variety. Thus 
\begin{defn}
The essential dimension of a finite group $G$ is
\[
\ed(G)=\ed(V)
\]
where $V$ is a faithfull linear $G-$variety.
\end{defn}

\subsection{Finite Abelian Groups}
Now we come to the result on finite abelian groups mentioned in the introduction:
\begin{thm}{\rm(Buhler-Reichstein)}\label{fag}
Let $G$ be a finite abelian group of rank $n$. Then $\ed(G)=n$.
\end{thm}
\begin{proof}G is finite and of rank $n$ so
\[
G\cong \zz/d_1\times...\times\zz/d_n\quad \mbox{for integers }d_1,...,d_n\ge 2 
\]
\[
\mbox{and }\;d_1\mid\cdots\mid d_n
\]
We will identify the two groups.
$G$ acts faithfully on $X=\mathbb{A}^n$ by 
$$(z_1,...,z_n)\cdot(x_1,...,x_n)=(\zeta_1^{z_1}x_1,...,\zeta_n^{z_n}x_n)$$
where $\zeta_i$ is a $d_i$th root of unity ($\ne1$). In other words, $G$ has a faithful linear representation of  dimension $n$ and so $\ed(G)\le n$.
\\
Suppose now we have a rational dominant $G$-equivariant map $f:X\dashrightarrow Y$ onto a $G$-variety $Y$ (i.e. $Y$ is a compression) such that $\ed(G)=\dim(Y)$. Define a valuation $\upsilon:k(x_1,...,x_n)\rightarrow \zz^n$ as in the standart example 2.4. through subvarieties $H_1=\{x_1=0\}, H_2=\{x_1=0,x_2=0\},
...,H_n=\{x_1=0,...,x_n=0\}.$ It will turn out that the induced valuation on $k(Y)$ must be of rank $n$ which gives the desired lower bound on the essential dimension. Let $p$ be a prime integer that divides $d_1$ (hence divides all $d_i$).\\
If $f_i=r/q=rq^{p-1}/q^{p}$ with polynomials $r,q$, we have
\[
\upsilon(f_i)=\upsilon(rq^{p-1})-p\upsilon(q)= \upsilon(rq^{p-1})\mod p,
\]
so if we are only interested in the valuation mod $p$, we can assume that $f_i$ is a polynomial in $x_1,...,x_n$.
Let $g=(z_1,...,z_n)\in G$. From the $G-$\nolinebreak equivariancy of $f$ we get
\[
f_j(\zeta_1^{z_1}x_1,...,\zeta_n^{z_n}x_n)=f_j(gx)=(gf(x))_j=\zeta_j^{z_j}f_j(x)
\]
and so every term in $f_j$ must be of the form
$x_1^{e_1}\cdots x_j^{e_j}\cdots x_n^{e_n}, \quad e_j\equiv 1 \mod d_j,$ 
$e_i\equiv 0 \mod d_j, i\ne j$ and the same holds$\mod p$.
Thus $\upsilon(f_j)\mod p =(0,..,1,..,0)$ and 
\[
\upsilon(f)\mod p=\left(\begin{array}{c}\upsilon(f_1)\\:\\ \upsilon(f_r)\end{array}\right)\mod p=I.
\]
Lemma \ref{rk1} asserts that $\rank \upsilon(f)=n$ and from Proposition \ref{pfg} it follows that $\ed(G)=\dim(Y)\ge n$.
\end{proof}
\begin{rem}
It is not known what the essential dimension of an arbitrary (non abelian) finite group is. For example the essential dimension of the symmetric group $S_n$, which is of great interest, is only known for $n\le 6$.
\end{rem}

\section{Canonical Dimension}\label{cd}
The purpose of this section is to recall the notion of canonical dimension and show its connection to the essential dimension of the functor of orbits (of the action of an algebraic group). In particular, we will obtain the result of Theorem \ref{edgen}, the essential dimension of the generic orbit  of a torus action, using the theory of canonical dimension. Canonical dimension was introduced by G. Berhuy and Z. Reichstein in \cite{BR3}. We refer also to \cite{Re1} for more information on the essential dimension of algebraic groups.

Let $G$ be an algebraic group and $X$ an irreducible variety on which $G$ acts.
\begin{defn}
A  variety  with function field $k(X)^G$ is called a \textit{rational quotient} of $X$ and denoted $X/G$. There exists a rational map $\pi:X\dashrightarrow X/G$ that induces the inclusion $k(X/G)=k(X)^G\hookrightarrow k(X)$.
\end{defn}
Note: $X/G$ and $\pi$ are only defined up to birational isomorphism.
\begin{defn}
A \textit{canonical form map} $F:X\dashrightarrow X$ is a rational map such that $F(x)=g(x)x \;\forall x\in X$ where $g:X\dashrightarrow G$ is some rational map.
\end{defn}
The canonical dimension is defined as
\begin{defn}
The \textit{canonical dimension} of a $G-$variety $X$ is 
\[
\cd(X,G)=\min\{\dim F(X)\}-\dim(X/G)
\]
where $F$ ranges over all canonical form maps.
\end{defn}
Now we outline some results that connect canonical dimension, essential dimension of orbits of a torus action and essential dimension of finite Abelian groups.\\
As in section \ref{ota}, let $T(K)=(K^*)^m$ be a torus acting on $X(K)=K^n$ through characters and $E$ be the matrix that determines the action. Once again, we are interested in the essential dimension of the generic point $[x]$.
We can assume that the matrix of exponents is in Smith normal form, so that the characters are simply
\[
\chi_i(t)=\left\{\begin{array}{ll}t_i^{d_i}&i\le r\\1&i>r\end{array}\right.
\]
(see (\ref{rac})) and without loss of generality, assume that $E\in M_{n,r}$ i.e. $E$ has full rank. 
From \cite[Prop. 14.1]{BR3} we get the following equality
\begin{equation}
\ed[x]=\cd(X,T)+\dim X/T 
\end{equation}
where $x$ is the generic point of $X$.
Let $S\subset T$ be the kernel of the action of $T$ on $X$,
\[
S=\{t\in T\mid\chi(t)=1\}\cong\zz/d_1\times...\times\zz/d_r
\]
The quotient $T/S$ is connected and the induced representation of $T/S$ on $X$ is diagonalizable hence $T/S$ is a torus. The essential dimension of a torus is $0$ (\cite[Ex. 3.9]{Re1}) and so from \cite[Lemma 5.2]{Re1},\cite[1.4.1]{Po} $X$ splits
\begin{equation}
X \mbox{ is birationally isomorphic to }T/S\times X/(T/S)
\end{equation}
$T$ acts trivially on $X/(T/S)$ and applying \cite[Lemma 7.1]{BR3} we have
\begin{equation}
\cd(X,T)=\cd(T/S\times X/(T/S),T)=\cd(T/S,T)
\end{equation}
\begin{defn}
A subgroup $H$ of a group $G$ that acts on a variety $X$ is called a \textit{stabilizer in general position} if for every point $x$ in general position the stabilizer Stab$(x)$ is conjugated to $H$.
\end{defn}
$S$ is the stabilizer of every element $\in(k^*)^n\subset X$ so $S$ is a stabilizer in general position. 
\cite[Prop. 5.7.b]{BR3} and \cite[Prop. 5.3]{Re1} then assert that
\begin{equation}
\cd(T/S,T)=\ed(S)
\end{equation}
$S$ is a finite Abelian group and so the essential dimension is its rank (Theorem \ref{fag}), $\ed(S)=\rank (S)=r-$\#ones among $\{d_1,...,d_r\}$.   
\begin{lem}
$\dim X/T=n-r$
\end{lem}
\begin{proof}
We show that the coordinate functions $x_{r+1},...,x_n$ generate $k(X)^G$.
Clearly $x_{r+1},...,x_n$ are algebraically independent and invariant under $T$, $k(x_{r+1},...,x_n)\subset k(X)^T$. Let $a,b\in X_0=(k^*)^n\subset X$ with $a_{r+1}=b_{r+1},...,a_n=b_n$ i.e. $x_{r+1},...,x_n$ assume the same values on $a$ and $b$. We can take $t_i=(b_i/a_i)^{1/d_i}$ (any $d_i$th root) for $i=1,...,r$ to see that $b$ is in the orbit of $a$. Hence, $x_{r+1},...,x_n$ separate orbits in general position and so $x_{r+1},...,x_n$ generate $k(X)^T$ by \cite[Lemma 2.1]{PV} and thus $\dim X/T=n-r$.
\end{proof}
Collecting all these results we finally get
\[
\ed[x]=\cd(X,T)+\dim X/T=n-\mbox{\#ones among }\{d_1,...,d_r\}
\]
as expected.

\end{document}